\documentclass{article}
\usepackage{amssymb,amsfonts,amsmath,amsthm,amsopn,amstext,amscd,latexsym,xy,mathdots,stmaryrd}
\usepackage{hyperref}
 \DeclareFontFamily{U}{wncy}{}
\DeclareFontShape{U}{wncy}{m}{n}{<->wncyr10}{}
\DeclareSymbolFont{mcy}{U}{wncy}{m}{n}
\DeclareMathSymbol{\Sha}{\mathord}{mcy}{"58} 
\theoremstyle{plain}
\input xy
\xyoption{all}
\newtheorem{theorem}{Theorem}
\newtheorem{lemma}[theorem]{Lemma}
\newtheorem{proposition}[theorem]{Proposition}

\theoremstyle{remark}

\numberwithin{equation}{section}
\numberwithin{paragraph}{section}

\DeclareMathOperator{\ord}{ord}

\DeclareMathOperator{\Sel}{Sel}

\newcommand{\cL}{{\mathcal L}}

\newcommand{\bbA}{{\mathbb A}}

\newcommand{\bbF}{{\mathbb F}}

\newcommand{\bbQ}{{\mathbb Q}}

\newcommand{\bbZ}{{\mathbb Z}}

\newcommand{\GL}{\mathrm{GL}}

\title{Elliptic curves over $\bbQ_{\infty}$ are modular}

\author{Jack A. Thorne\footnote{\textsc{Department of Pure Mathematics and Mathematical Statistics, Wilberforce Road, Cambridge, United Kingdom.} \textit{Email address:} \texttt{thorne@dpmms.cam.ac.uk}}} 

\begin{document}
\maketitle
\begin{abstract}
We show that if $p$ is a prime, then all elliptic curves defined over the cyclotomic $\bbZ_p$-extension of $\bbQ$ are modular.
\end{abstract}

\section{Introduction}

Our goal in this paper is to prove the following theorem:
\begin{theorem}\label{thm_main_theorem}
Let $p$ be a prime, and let $F$ be a number field which is contained in the cyclotomic $\bbZ_p$-extension of $\bbQ$. Let $E$ be an elliptic curve over $F$. Then $E$ is modular.
\end{theorem}
We recall that an elliptic curve $E$ over a number field $F$ is said to be modular if there is a regular algebraic automorphic representation $\pi$ of $\GL_2(\bbA_F)$ which has the same $L$-function as $E$. This is one of several equivalent formulations; if $F$ is totally real, then $\pi$ will be generated by vectors which can be interpreted as Hilbert modular forms of parallel weight 2.

The modularity of a given elliptic curve $E$ has many useful consequences. It implies that the $L$-function of $E$ has an analytic continuation to the whole complex plane, allowing one to formulate the Birch--Swinnerton-Dyer conjecture for $E$ unconditionally. When the order of vanishing of the $L$-function at the point $s = 1$ is at most 1, this conjecture is almost known in many cases \cite{Zha01}.

The modularity of all elliptic curves over $\bbQ$ has been known since work of Wiles and Breuil, Conrad, Diamond, and Taylor \cite{Wil95}, \cite{Tay95}, \cite{Con99}, \cite{Bre01}. Attempts to generalize this work to fields other than $\bbQ$ have all followed Wiles' original strategy: one first proves automorphy lifting theorems. For curves satisfying the conditions of these theorems, one attempts to verify the residual automorphy. One then hopes that curves not satisfying the conditions of these theorems can be enumerated and checked explicitly to be modular.

This strategy has recently been used by Freitas, Le Hung, and Siksek to establish the modularity of all elliptic curves over real quadratic fields \cite{Fre13}. They use automorphy lifting theorems of Kisin to reduce the result to a calculation of real quadratic points on a finite list of modular curves; they then carry out some formidable calculations to check that these points also correspond to modular elliptic curves. 

In this paper, we use an automorphy lifting theorem established recently by the author which extends the domain of validity of Kisin's theorems \cite{Tho15}, and reduces the modularity of elliptic curves over a given totally real field $F$ with $\sqrt{5} \not\in F$ to checking the modularity of elliptic curves corresponding to rational points on two modular curves (one of which is $X_0(15)$, and the other of which is a genus 1 curve isogenous to $X_0(15)$). We then use Iwasawa theory to check that these curves acquire no new rational points in any cyclotomic $\bbZ_p$-extension of $\bbQ$.

\subsection*{Acknowledgements} This work was carried out while the author served as a Clay Research Fellow. I would like to thank James Newton and John Coates for useful conversations.

\section{The proof}

If $F$ is a number field, we write $G_F$ for its absolute Galois group (relative to a fixed choice of algebraic closure). If $E$ is an elliptic curve defined over a number field $F$, and $p$ is a prime, then we write $\rho_{E, p} : G_F \to \GL_2(\bbZ_p)$ for the associated representation of the absolute Galois group of $F$ on the $p$-adic Tate module of $E$, and $\overline{\rho}_{E, p} : G_F \to \GL_2(\bbF_p)$ for its reduction modulo $p$. Both of these representations are defined up to conjugation by elements of $\GL_2(\bbZ_p)$. We write $\zeta_p$ for a choice of primitive $p^\text{th}$ root of unity in the algebraic closure of $F$.
\begin{theorem}\label{thm_modularity_lifting_results}
Let $E$ be an elliptic curve over a totally real number field $F$, and suppose that (at least) one of the following is true:
\begin{enumerate}
\item The representation $\overline{\rho}_{E, 3}|_{G_{F(\zeta_3)}}$ is absolutely irreducible. 
\item $\sqrt{5} \not\in F$, and $\overline{\rho}_{E, 5}$ is irreducible.
\end{enumerate}
Then $E$ is modular.
\end{theorem}
\begin{proof}
The first part follows from results of Kisin and Langlands--Tunnell, see \cite[Theorem 3]{Fre13}. The second part, in the case where $\overline{\rho}_{E, 5}$ remains absolutely irreducible on restriction to $G_{F(\zeta_5)}$, is a consequence of the first part and the 3--5 switch of Wiles, described in \emph{loc. cit.}. The second part in the remaining case is \cite[Theorem 1.1]{Tho15}. 
\end{proof}
Following \cite[\S 2.2]{Fre13}, we introduce modular curves $X(s3, b5)$ and $X(b3, b5)$. These are smooth, projective curves over $\bbQ$. Loosely speaking, for a number field $K$ the non-cuspidal $K$-points of $X(s3, b5)$ correspond to isomorphism classes of elliptic curves $E$ such that $\overline{\rho}_{E, 3}(G_K)$ is contained in the normalizer of a split Cartan subgroup of $\GL_2(\bbF_3)$ and $\overline{\rho}_{E, 5}(G_K)$ is contained in a Borel subgroup of $\GL_2(\bbF_5)$ (i.e.\ $E$ has a $K$-rational $5$-isogeny). The non-cuspidal $K$-points of $X(b3, b5)$ correspond to isomorphism classes of curves which are endowed with a $K$-rational $15$-isogeny. 
\begin{lemma}\label{lem_modularity_criterion}
Let $F$ be a totally real field such that $\sqrt{5} \not\in F$. 
\begin{enumerate}
\item If $E$ is an elliptic curve over $F$ which is not modular, then $E$ determines an $F$-rational point of one of the curves $X(s3, b5)$, $X(b3, b5)$.
\item If $F / \bbQ$ is cyclic and $X(s3, b5)(F) = X(s3, b5)(\bbQ)$, $X(b3, b5)(F) = X(b3, b5)(\bbQ)$, then all elliptic curves over $F$ are modular.
\end{enumerate}
\end{lemma}
\begin{proof}
The first part is a consequence of Theorem \ref{thm_modularity_lifting_results} and \cite[Proposition 4.1]{Fre13}. The second part is a consequence of the first part, the modularity of all elliptic curves over $\bbQ$, and cyclic base change for $\GL_2$ \cite{Lan80}.
\end{proof}
We can make these modular curves explicit:
\begin{proposition}
The curve $X(b3, b5)$ is isomorphic over $\bbQ$ to the elliptic curve
\[ E_1 : y^2 + xy + y = x^3 + x^2 - 10 x - 10 \]
of Cremona label 15A1. The curve $X(s3, b5)$ is isomorphic over $\bbQ$ to the elliptic curve
\[ E_2 : y^2 + xy + y = x^3 + x^2 - 5 x + 2 \]
of Cremona label 15A3. Both curves have group of rational points isomorphic to $\bbZ / 2 \bbZ \oplus \bbZ / 4 \bbZ$. They are related by an isogeny of degree 2.
\end{proposition}
\begin{proof}
See \cite[Lemmas 5.6, 5.7]{Fre13}. 
\end{proof}
By Lemma \ref{lem_modularity_criterion}, the proof of Theorem \ref{thm_main_theorem} is therefore reduced to the following result:
\begin{theorem}
Let $p$ be a prime, let $i \in \{1, 2 \}$, and let $\bbQ_{\infty}$ denote the cyclotomic $\bbZ_p$-extension of $\bbQ$. Then $E_i(\bbQ_{\infty}) = E_i(\bbQ)$.
\end{theorem}
\begin{proof}
We just treat the case of $E = E_1$, since the other case is extremely similar (because $E_1, E_2$ are related by a 2-isogeny). We first show $E(\bbQ_{\infty})^\text{tors} = E(\bbQ)$ by studying the Galois representations of $E$. It is known that the Galois representation $\rho_{E, l} : G_\bbQ \to \GL_2(\bbZ_l)$ is surjective for all primes $l \geq 3$. (For example, this can be shown using \cite[Proposition 21]{Ser72} and the fact that $E$ has minimal discriminant $15^4$.) It follows that for any prime $l \geq 3$, we have $E(\bbQ_{\infty})[l^\infty] = 0$. To address the 2-power torsion, we appeal to the calculation by Rouse and Zureick-Brown of the image $\rho_{E, 2}(G_\bbQ)$ \cite{Zur15}. They show that this subgroup of $\GL_2(\bbZ_2)$ is the pre-image of the subgroup $G \subset \GL_2(\bbZ / 8 \bbZ)$ which is (up to conjugation) generated by the following matrices\footnote{See \texttt{http://users.wfu.edu/rouseja/2adic/X187d.html}.}:
\[ G = \left\langle \left(\begin{array}{cc} 5 & 4 \\ 2 & 3 \end{array} \right), \left(\begin{array}{cc} 1 & 0 \\ 0 & 5 \end{array} \right), \left(\begin{array}{cc} 1 & 4 \\ 0 & 5 \end{array} \right), \left(\begin{array}{cc} 1 & 0 \\ 4 & 5 \end{array} \right) \right\rangle. \]
This group has order 16. In particular, $\rho_{E, 2}(G_\bbQ)$ is a pro-2 group, and if $p > 2$ then $\rho_{E, 2}(G_{\bbQ_\infty}) = \rho_{E, 2}(G_\bbQ)$ and hence $E(\bbQ_\infty)[2^\infty] = E(\bbQ)[2^\infty]$. If $p = 2$, then $\rho_{E, 2}(G_{\bbQ_\infty}) = \{ g \in \rho_{E, 2}(G_\bbQ) \mid \det(g)^2 = 1\}$. We calculate that the subgroup of $G$ consisting of matrices with determinant $\pm 1$ is given by
\[ H = \left\langle \left(\begin{array}{cc} 5 & 4 \\ 2 & 3 \end{array} \right), \left(\begin{array}{cc} 5 & 0 \\ 2 & 3 \end{array} \right), \left(\begin{array}{cc} 1 & 0 \\ 4 & 1 \end{array} \right) \right\rangle. \]
We have $(\bbZ / 8 \bbZ \times \bbZ / 8 \bbZ)^H = (\bbZ / 8 \bbZ \times \bbZ / 8 \bbZ)^G$, and this easily implies that $E(\bbQ_\infty)[2^\infty] = E(\bbQ)[2^\infty]$ in this case also. (We checked these group-theoretic calculations using Magma \cite{Bos97}.) Since the group $E(\bbQ_{\infty})^\text{tors}$ is the product of its $l$-primary components, we have now shown it to be equal to $E(\bbQ)$.

We now show that the group $E(\bbQ_{\infty})$ is finite, using results in Iwasawa theory. Since we already know $E(\bbQ_{\infty})^\text{tors} = E(\bbQ)$, this will prove the theorem. Calculations of Greenberg \cite[p. 136]{Gre99} in the case $p = 2$ show that the $2$-adic $\lambda$-invariant of $E$ is trivial, hence $E(\bbQ_{\infty})$ is finite. We can therefore assume that $p \geq 3$. It is known (cf. the tables in \cite{Cre97}) that the $L$-function of $E$ satisfies $L(E, 1) / \Omega_E = 1/8$. The product $\text{Tam}(E)$ of the Tamagawa numbers of $E$ is equal to $8$.

If $p \geq 3$ is a prime of good ordinary reduction for $E$, then by \cite[Proposition 3.8]{Gre99}, to show  $E(\bbQ_{\infty})$ is finite it is enough to show that $\Sel_p(E) = \Sha(E)[p]$ is trivial and $a_p \not \equiv 1 \text{ mod }p$. We have $a_p \not\equiv 1 \text{ mod }p$ if and only if $p$ divides the group $E(\bbF_p)$. If this happens then $8p$ divides $E(\bbF_p)$, which contradicts the Hasse bound. The triviality of $\Sel_p(E)$ follows from results of Kato and the fact that $L(E, 1) / \Omega_E$ is a $p$-adic unit, see \cite[Theorem 17.4]{Kat04} or \cite[Theorem 3.35]{Ski14} for a convenient reference. If $p$ is a prime of good supersingular reduction, then the finiteness of  $E(\bbQ_{\infty})$ follows from results of Kato and Kurihara \cite[Theorem 0.1]{Kur02} and the fact that $L(E, 1) / \Omega_E$ is a $p$-adic unit. It remains to treat the primes $p = 3, 5$ of bad reduction of $E$. In either case, it follows from \cite[Theorem C]{Ski15} and the fact that $L(E, 1) / \Omega_E$ is a $p$-adic unit that $\Sel_p(E)$ is trivial. The finiteness of $E(\bbQ_\infty)$ then follows from the discussion on \cite[pp. 92--93]{Gre99}, on noting that $E$ has non-split multiplicative reduction at $p = 3$, and split multiplicative reduction at $p = 5$, where the $\cL$-invariant $\frac{\log_p q_E}{\ord_p q_E}$ lies in $p \bbZ_p^\times$. This completes the proof.
\end{proof}

\def\cprime{$'$}


\begin{thebibliography}{BCDT01}

\bibitem[BCDT01]{Bre01}
Christophe Breuil, Brian Conrad, Fred Diamond, and Richard Taylor.
\newblock On the modularity of elliptic curves over {$\bold Q$}: wild 3-adic
  exercises.
\newblock {\em J. Amer. Math. Soc.}, 14(4):843--939 (electronic), 2001.

\bibitem[BCP97]{Bos97}
Wieb Bosma, John Cannon, and Catherine Playoust.
\newblock The {M}agma algebra system. {I}. {T}he user language.
\newblock {\em J. Symbolic Comput.}, 24(3-4):235--265, 1997.
\newblock Computational algebra and number theory (London, 1993).

\bibitem[CDT99]{Con99}
Brian Conrad, Fred Diamond, and Richard Taylor.
\newblock Modularity of certain potentially {B}arsotti-{T}ate {G}alois
  representations.
\newblock {\em J. Amer. Math. Soc.}, 12(2):521--567, 1999.

\bibitem[Cre97]{Cre97}
J.~E. Cremona.
\newblock {\em Algorithms for modular elliptic curves}.
\newblock Cambridge University Press, Cambridge, second edition, 1997.

\bibitem[FLHS]{Fre13}
Nuno Freitas, Bao~V. Le~Hung, and Samir Siksek.
\newblock Elliptic curves over real quadratic fields are modular.
\newblock Preprint. Available at {\tt http://arxiv.org/abs/1310.7088}.

\bibitem[Gre99]{Gre99}
Ralph Greenberg.
\newblock Iwasawa theory for elliptic curves.
\newblock In {\em Arithmetic theory of elliptic curves ({C}etraro, 1997)},
  volume 1716 of {\em Lecture Notes in Math.}, pages 51--144. Springer, Berlin,
  1999.

\bibitem[Kat04]{Kat04}
Kazuya Kato.
\newblock {$p$}-adic {H}odge theory and values of zeta functions of modular
  forms.
\newblock {\em Ast\'erisque}, (295):ix, 117--290, 2004.
\newblock Cohomologies $p$-adiques et applications arithm{\'e}tiques. III.

\bibitem[Kur02]{Kur02}
Masato Kurihara.
\newblock On the {T}ate {S}hafarevich groups over cyclotomic fields of an
  elliptic curve with supersingular reduction. {I}.
\newblock {\em Invent. Math.}, 149(1):195--224, 2002.

\bibitem[Lan80]{Lan80}
Robert~P. Langlands.
\newblock {\em Base change for {${\rm GL}(2)$}}, volume~96 of {\em Annals of
  Mathematics Studies}.
\newblock Princeton University Press, Princeton, N.J., 1980.

\bibitem[RZB]{Zur15}
Jeremy Rouse and David Zureick-Brown.
\newblock Elliptic curves over {$\mathbb{Q}$} and 2-adic images of {G}alois.
\newblock Preprint. Available at \texttt{http://arxiv.org/abs/1402.5997}.

\bibitem[Ser72]{Ser72}
Jean-Pierre Serre.
\newblock Propri\'et\'es galoisiennes des points d'ordre fini des courbes
  elliptiques.
\newblock {\em Invent. Math.}, 15(4):259--331, 1972.

\bibitem[Ski]{Ski15}
Christopher Skinner.
\newblock Multiplicative reduction and the cyclotomic main conjecture for
  {$\mathrm{GL}_2$}.
\newblock Preprint, avaiable at \texttt{http://arxiv.org/abs/1407.1093}.

\bibitem[SU14]{Ski14}
Christopher Skinner and Eric Urban.
\newblock The {I}wasawa main conjectures for {$\rm GL_2$}.
\newblock {\em Invent. Math.}, 195(1):1--277, 2014.

\bibitem[Tho]{Tho15}
Jack~A. Thorne.
\newblock Automorphy of some residually dihedral {G}alois representations.
\newblock To appear in Math. Annalen.

\bibitem[TW95]{Tay95}
Richard Taylor and Andrew Wiles.
\newblock Ring-theoretic properties of certain {H}ecke algebras.
\newblock {\em Ann. of Math. (2)}, 141(3):553--572, 1995.

\bibitem[Wil95]{Wil95}
Andrew Wiles.
\newblock Modular elliptic curves and {F}ermat's last theorem.
\newblock {\em Ann. of Math. (2)}, 141(3):443--551, 1995.

\bibitem[Zha01]{Zha01}
Shouwu Zhang.
\newblock Heights of {H}eegner points on {S}himura curves.
\newblock {\em Ann. of Math. (2)}, 153(1):27--147, 2001.

\end{thebibliography}
\end{document}